\documentclass[10pt,a4paper,twoside]{amsart}
\usepackage[T1]{fontenc}
\usepackage[utf8]{inputenc}
\usepackage[english]{babel}
\usepackage{amsmath}
\usepackage{amsthm}
\usepackage{amssymb}
\usepackage{amsfonts}
\usepackage{amsxtra}
\usepackage{enumerate}
\usepackage{verbatim}
\usepackage{color}
\usepackage[mathscr]{eucal}
\usepackage{graphicx}
\usepackage{xcolor}
\usepackage[backref=page]{hyperref}
\usepackage{tikz}
\usepackage{enumitem}
\usepackage{parskip}

\newcommand{\R}{\mathbb R}

\newtheorem*{nota}{Notation}
\newtheorem{theorem}{Theorem}[section]
\newtheorem{definition}[theorem]{Definition}
\newtheorem{lemma}[theorem]{Lemma}

\newtheorem{prop}[theorem]{Proposition}
\newtheorem{conjecture}[theorem]{Conjecture}

\newtheorem{remark}[theorem]{Remark}

\theoremstyle{definition}
\newtheorem{exmp}[theorem]{Example}

\let\phi=\varphi

\newcommand{\integ}[4]{\int\limits_{#1}^{#2}#3\, d#4}
\newcommand{\summe}[3]{\sum\limits_{#1}^{#2}#3}
\newcommand{\abb}[3]{#1\colon #2\rightarrow #3}
\newcommand{\real}[1]{\mathbb{R}^{#1}}
\newcommand{\rem}[1]{}

\DeclareFontFamily{U}{mathb}{\hyphenchar\font45}
\DeclareFontShape{U}{mathb}{m}{n}{
<-6> mathb5 <6-7> mathb6 <7-8> mathb7
<8-9> mathb8 <9-10> mathb9
<10-12> mathb10 <12-> mathb12
}{}
\DeclareSymbolFont{mathb}{U}{mathb}{m}{n}
\DeclareMathSymbol{\llcurly}{\mathrel}{mathb}{"CE}
\DeclareMathSymbol{\ggcurly}{\mathrel}{mathb}{"CF}

\title[On the space of cone geodesics]{On the space of cone geodesics and positive paths of contactomorphisms}

\author{Jakob Hedicke}
\address{ Radboud Universiteit Nijmegen, Heyendaalseweg 135, 6525 AJ NIJMEGEN, The Netherlands} 
\email{jakob.hedicke@gmail.com}

\date{\today}
\begin{document}
\begin{abstract}
Often it is possible to equip the space of all cone geodesics of a strongly convex cone structure with the structure of a smooth contact manifold.
This generalizes the analogous notions for the space of light rays of a Lorentzian spacetime.
After reviewing these constructions on the space of cone geodesics, with a focus on the natural contact structure, we establish a correspondence between positive paths of contactomorphisms in spherical cotangent bundles and certain globally hyperbolic cone structures.

\end{abstract}

\maketitle

\section{Introduction}

In his seminal works on twistor theory Roger Penrose observed that in many physically relevant situations the space of all light rays, i.e., of unparametrized null geodesics of a Lorentzian spacetime, is a smooth manifold that can be naturally equipped with a contact structure, see e.g. \cite{Penrose682, Penrose722, Penrose84}.
Since then the space of light rays has been studied by various authors from a relativistic \cite{Low89, Low06, Bautista14, Bautista17, Hedicke20} and from a contact geometric point of view \cite{Natario04, Chernov10, Chernov102, Hedicke21, Hedicke24, Marin24}.
These results indicate a deep connection between topological and geometric properties of the space of light rays and the causality of the underlying spacetime.
For example, the existence of a smooth structure (possibly non-Hausdorff) on the space of light rays follows from the strong causality of the spacetime \cite{Low89}.
Further results relate the existence of a smooth manifold structure and a natural contact structure to global hyperbolicity \cite{Low06} and causal simplicity \cite{Hedicke20}.
In the case of globally hyperbolic spacetimes the relativistic structure of the spacetime can be related to various notions from contact geometry.
For instance causality is connected to Legendrian linking \cite{Natario04, Chernov10} or the positivity of Legendrian isotopies \cite{Chernov10, Chernov102} in the space of light rays, and the gravitational redshift can be described using cotact forms \cite{Chernov18}.

Recent developments in general relativity lead to generalized notions of spacetimes, such as manifolds with cone structures \cite{Bernard18, Minguzzi192} and Lorentz-Finsler spacetimes \cite{Beem70, Asanov85, Javaloyes20}.
In the setting of cone structures one can define cone geodesics, which are a direct generalization of light rays in Lorentzian manifolds.
For strongly convex cone structures a construction analogous to the case of Lorentzian metrics leads to a space of cone geodesics with similar geometrical and topological properties \cite{Hedickediss, Herrera25, Sanchez25}.

In this note we review the construction of the space of cone geodesics for smooth strongly convex cone structures, with a particular focus on the contact geometric point of view.
We furthermore establish a correspondence between globally hyperbolic cone structures and positive paths of contactomorphisms in spherical cotangent bundles.

\subsection*{Summary of the paper}

In Section 2 we start with reviewing some basic facts about closed cone structures on smooth manifolds and their causality, as well as about Lorentz-Finsler metrics and their geodesics.
This allows to define the notion of cone geodesics and, in the case of certain strongly convex cone structures, to equip the space of all cone geodesics with the structure of a smooth manifold.
We proceed by recalling basic notions from contact geometry, and show that the space of cone geodesics can be equipped with a natural cooriented contact structure.
Finally, we present a proof that in the globally hyperbolic case the space of cone geodesics is contactomorphic to the spherical cotangent bundle of its Cauchy hypersurfaces.

The considerations of Section 3 can be summarized in the following Theorems, establishing a correspondence between certain globally hyperbolic cone structures and positive paths of contactomorphisms.

\begin{theorem}\label{thm1}
Let $(\R\times\Sigma,C)$ be a globally hyperbolic strongly convex cone structure such that the projection to the first coordinate is a Cauchy time function.
Then there exists a positive path of contactomorphism $(\phi_t^C)_{t\in\R}$ of the spherical cotangent bundle $ST^{\ast}\Sigma$, such that up to re-parametrization all cone geodesics are of the form $t\mapsto (t,\pi(\phi^C_t(v)))$.
Here $v\in ST^{\ast}\Sigma$ and $\pi\colon ST^{\ast}\Sigma\rightarrow \Sigma$ denotes the canonical projection.
\end{theorem}

Theorem \ref{thm1} is proved in Section 3.1.

In Section 3.2 we prove the following.

\begin{theorem}\label{thm2}
Let $(f_t)_{\in \R}$ be a positive path of contactomorphisms of $ST^{\ast}\Sigma$.
Then $(f_t)_{\in \R}$ defines a cone structure $C_f$ on $\R\times \Sigma$ that can be naturally equipped with the structure of a locally Lipschitz Lorentz-Finsler space.

If $(f_t)_{\in \R}=(\phi^C_t)_{\in \R}$ is the positive path induced by a globally hyperbolic strongly convex cone structure $(\R\times \Sigma,C)$, then $C_{\phi^C}=C$.
\end{theorem}

\subsection*{Acknowledgements}
Many of ideas in this work have been part of my PhD thesis, which was written under the supervision of Stefan Nemirovski and Stefan Suhr and supported by the SFB/TRR 191 “Symplectic Structures in Geometry, Algebra
and Dynamics”.
Further I would like to thank Alberto Abbondandolo for valuable discussions about convex sets.
This work was supported by a Radboud Excellence Fellowship.

\section{Preliminaries}\label{c1}

\subsection{Cone structures and causality}

The notion of cone structure on a smooth manifold $M$ generalizes (time-oriented) Lorentzian metrics by looking at distributions of convex cones in the tangent bundle of $M$ that mimic the properties of the future light cones in the Lorentzian case.
Cone structures can be defined in various levels of generality \cite{Bernard18, Minguzzi192, Javaloyes20}.
In this paper we mostly consider cones of the following type.

\begin{definition}\label{defcone}
A \textbf{closed cone} in a vector space $V$ is a closed subset $C\subset V\setminus\{0\}$ such that 
\begin{itemize}
\item[(i)]$\lambda v \in C$ for all $v\in C$ and $\lambda>0$.
\item[(ii)]If $v\in C$ then $-v\notin C$.
\end{itemize}
The cone $C$ is called \textbf{proper} if $C$ is convex and has a non-empty interior.
A proper cone $C$ is called \textbf{strongly convex} if additionally $\partial C\setminus\{0\}$ is smooth and
\begin{itemize}
\item[(iii)] The second fundamental form of $\partial C$ with respect to an outward pointing direction is positive semi-definite and its kernel at a point $v\in \partial C$ is given by $\{\lambda v|\lambda>0\}$.
\end{itemize}
\end{definition}

\begin{definition}\label{defconestr}
A \textbf{closed cone structure} on a manifold $M$ is a subset $C\subset TM$ such that $C_p:=C\cap T_pM$ is a proper closed cone for all $p\in M$.
It is called \textbf{strongly convex} if $\partial C$ is an embedded hypersurface transverse to the fibres of $TM\setminus\{0\}$ such that $C_p$ is a strongly convex cone for every $p\in M$.
\end{definition}

\begin{nota}
Throughout this paper $\partial C$ will always denote the boundary of $C$ in $TM\setminus \{0\}$, i.e., $\partial C$ does not intersect the $0$ section.
\end{nota}

Analogously to the case of time-oriented Lorentzian metrics one can define the notion of future pointing timelike, causal and null vectors for cone structures.
We call $v\in C^{\circ}$ \textbf{future pointing timelike}, $v\in C$ \textbf{future pointing causal} and $v\in \partial C$ \textbf{future pointing null}.
Moreover, we call $v$ past pointing timelike/causal/null if $-v$ is future pointing timelike/causal/null.
Here $\partial C$ always denotes the boundary of $C$ in $TM\setminus\{0\}$.
A curve $\gamma\colon I\rightarrow M$ is called future pointing timelike/causal/null if $\gamma'(t)$ has this property for all $t\in I$. 
As in the Lorentzian case this allows to define future chronological-, causal- and horismos relations:
\begin{align*}
I^+&:=\{(p,q)\in M\times M| \text{ there exists a future directed timelike curve from $p$ to $q$ }\}\\
J^+&:=\{(p,q)\in M\times M| \text{ there exists a future directed causal curve from $p$ to $q$ }\}\\
E^+&:= J^+\setminus I^+.
\end{align*}
Similarly one defines the past chronological-, causal- and horismos relations $I^-,J^-$ and $E^-$.
The chronological-, causal future/past and horismos of a point $p\in M$ are denoted by $I^{\pm}(p), J^{\pm}(p)$ and $E^{\pm}(p)$, respectively.

In this note we will mostly use the following steps from the causal hierarchy of cone structures.
Recall that a cone structure is \textbf{causal} if there are no future pointing causal loops in $(M,C)$.
For a detailed overview on the causality of cone structures, see \cite{Minguzzi192}.

\begin{definition}\label{defcausal}
A cone structure is called \textbf{strongly causal} if for every open set $U\subset M$ there exists a causally convex open subset $V\subset U$, i.e., a subset $V$ such that every causal curve with endpoints in $V$ is entirely contained in $V$.
A causal cone structure $(M,C)$ is called \textbf{causally simple} if the causal relation $J^+$ is closed as a subset of $M\times M$.
The cone structure $(M,C)$ is called \textbf{globally hyperbolic} if the causal diamonds $J^+(p)\cap J^-(q)\subset M$ are compact for all $p,q\in M$.
\end{definition}

An important characterization of globally hyperbolic Lorentzian spacetimes also holds for cone structures:
A cone structure $(M,C)$ is globally hyperbolic if and only if $M\cong \R\times \Sigma$ and $\{t\}\times \Sigma$ is a smooth \textbf{Cauchy hypersurface} for all $t\in\R$, i.e., a hypersurface transverse to $C$ such that every inextensible causal curve intersects $\Sigma$ in a unique point, see \cite[Theorem 2.42]{Minguzzi192}.

Note that in Lorentzian geometry null geodesics are a conformal invariant \cite{Beem96} and depend only on the light cones of the metric.
This allows to define the notion of cone geodesics for more general cone structures, without making use of a metric.
Cone geodesics are a generalization of light rays in classical spacetimes to the setting of cone structures.

\begin{definition}
A cone geodesic of a cone structure $(M,C)$ is a continuous curve $\abb{\gamma}{I}{M}$ that is locally horismotic.
This means that for every $t_0\in I$ and any open neighbourhood $U$ of $\gamma(t_0)$ there exists an open neighbourhood $\gamma(t_0)\in V\subset U$ such that if $\gamma([t_0-\epsilon,t_0+\epsilon]\cap I)\subset V$, then $\gamma(s_1)\in E_{C_U}^+(\gamma(s_0))$ for all $s_0,s_1\in [t_0-\epsilon,t_0+\epsilon]\cap I$ with $s_0<s_1$.
Here $C_U$ denotes the restriction of the cone structure $C$ to the neighbourhood $U$.
\end{definition}

\subsection{Lorentz-Finsler metrics and cone geodesics}\label{secFinsler}

Strongly convex cone structures naturally correspond to Lorentz-Finsler metrics (up to anisotropic equivalence), see \cite{Javaloyes20}.
Throughout this paper we mostly work with the classical definition of Lorentz-Finsler metrics as in \cite{Beem70}.

\begin{definition}\label{deffinsler}
A continuous function $\abb{L}{TM}{\real{}}$ that is smooth on $TM\setminus\{0\}$ is called a Lorentz-Finsler metric if
\begin{itemize}
\item[(i)] $L(\lambda v)=\lambda^2 v$ for all $v\in TM$ and all $\lambda\in \real{}$.
\item[(ii)] For every $v \in TM\setminus\{0\}$ its Fundamental tensor
$$g_v(u,w):=\frac{d^2}{dsdt}|_{t=s=0}L(v+su+tw)$$
has signature $(-,+,\cdots,+)$.
We call $(M,L)$ a Lorentz-Finsler spacetime.
\end{itemize}
\end{definition}

\begin{exmp}
Let $g$ be a Lorentzian metric on a manifold $M$.
Then the function $L(v):=g(v,v)$ defines a Lorentz-Finsler metric.
\end{exmp}

It follows from \cite[Corollary 5.8]{Javaloyes20} that for any strongly convex cone structure $C$ there exists a Lorentz-Finsler metric $L$ such that $\partial C$ is a connected component of $L^{-1}(0)\setminus\{0\}$.

Let $\abb{L}{TM}{\real{}}$ be a Lorentz-Finsler metric and $\abb{\gamma}{[a,b]}{M}$ be a piecewise smooth curve.
The \textbf{energy} of $\gamma$ with respect to $L$ is defined as
$$E_L(\gamma):=\frac{1}{2}\integ{a}{b}{L(\gamma'(t))}{t}.$$
The curve $\gamma$ is a \textbf{geodesic} of $L$ if it is a critical point of $E_L$ under variations with fixed end-points.

This can also be expressed in terms of the  \textbf{Chern connection}.
Given a nowhere vanishing vector field $V$ there exists a unique torsion free connection $\abb{\nabla^V}{\Gamma(TM)\times\Gamma(TM)}{\Gamma(TM)}$ satisfying 
$$X(g_{V}(Y,Z))=g_V(\nabla^V_XY,Z)+g_V(Y,\nabla^V_XZ)+2T_V(\nabla^V_XV,Y,Z).$$
Here $T$ denotes the Cartan tensor of $L$ (see e.g.\cite[Chapter 2]{Javaloyes15} for details).
This gives rise to a covariant derivative $\abb{D_{\gamma}^W}{\mathfrak{X}(\gamma)}{\mathfrak{X}(\gamma)}$ defined along smooth regular curves, where $\mathfrak{X}(\gamma)$ denotes the space of vector fields along $\gamma$.
Then geodesics are determined by the equation $D_{\gamma}^{\gamma'}\gamma'=0$ which in local coordinates reads as 
$$(\gamma^k)''+\summe{i,j=1}{n}{(\gamma^i)'(\gamma^j)'\Gamma_{ij}^{k}(\gamma')}=0,$$
where $\Gamma_{ij}^k$ are defined as usual by
$$\nabla^V_{\partial_i}\partial_j=:\summe{k=1}{n}{\Gamma_{ij}^{k}(V)\partial_k}.$$

As shown for example in \cite{Minguzzi14} in a more general setting, for any $v\in TM$ there exists a unique geodesic $\gamma_v$ with $\gamma'(0)=v$.
For any $p\in M$ this leads to a smooth \textbf{geodesic exponential function} $\abb{\exp_p}{V}{M}$ defined on a starshaped subset $V$ around $0\in T_pM$ by $\exp_p(v):=\gamma_v(1)$.
As shown in \cite[Proposition 3.10]{Suhr18} there exists a fibre-wise star-shaped open subset $D\subset TM$ of $0\in T_pM$ such that the map $\abb{\mathrm{Exp}}{D}{M\times M}$ defined by $\mathrm{Exp}(v,w):=(\pi(v),\gamma_w(1))$ becomes a diffeomorphism onto its image smooth outside of the $0$-section.
After eventually diminishing $D$ to a smaller subset $\tilde{D}$, $\exp_p(\tilde{D}\cap T_pM)$ defines a geodesically convex neighbourhood around $p$, i.e., any two points in this neighbourhood can be joined by a unique geodesic contained in $\exp_p(\tilde{D}\cap T_pM)$.

Similar to the case of pseudo-Riemannian metrics there exists a \textbf{geodesic spray}, i.e., a vector field $G$ on $TM$ whose integral curves project to geodesics on $M$, see e.g. \cite{Javaloyes21}.

For strongly convex cone structures it follows that cone geodesics coincide up to re-parametrization with the null geodesics of any Lorentz-Finsler metric defining the cone structure, see \cite{Javaloyes20}.

\subsection{The space of cone geodesics}

Since the work of Penrose \cite{Penrose68, Penrose722} and Low \cite{Low89, Low06} it is well-known that in many cases the space of light rays of a Lorentzian spacetime admits the structure of a smooth manifold.
The same is true for strongly convex cone structures, see \cite{Hedickediss, Herrera25}.
In this section we briefly recall the construction of the smooth structure on the space of cone geodesics.

Let $\mathcal{N}_C$ be the set of all inextensible cone geodesics of a strongly convex cone structure $(M,C)$ up to re-parametrization.
Here we view two cone geodesics as the same element in $\mathcal{N}_C$ iff their images coincide.
A natural topology on $\mathcal{N}_C$ can be defined as follows.
As explained in the previous section, every cone geodesic is, up to re-parametrization, a null geodesic for an auxiliary Lorentz-Finsler metric $L$ and therefore satisfies the geodesic equation.
Hence, as for Lorentzian spacetimes, the image of a cone geodesic is uniquely determined by a line in $\partial C$ that is collinear to the derivative of the cone geodesic at one point.
Obviously two lines in $\partial C$ define the same cone geodesic iff they are related by the geodesic spray of $L$.
This implies that $\mathcal{N}_C$ can be identified with a quotient of $\partial C$:
Let $G$ be the geodesic spray of $L$ and $Y$ the vector field inducing the canonical $\R$-action on $TM$ given by fibre-wise scalar multiplication with a positive scalar.
An easy computation shows that $[Y,G]=G$, i.e. the span of $G$ and $Y$ induces a foliation of $\partial C$ by $2$-dimensional leaves.
A leave consists of the image of a line in $\partial C$ under the geodesic flow of $L$.
In particular $\mathcal{N}_C$ can be naturally identified with the leaf-space of this foliation.
We equip $\mathcal{N}_C$ with the quotient topology.

\begin{nota}
Elements of $\mathcal{N}_C$ are equivalence classes of cone geodesics sharing the same image.
We will usually denote them by $[\gamma]$, where $\gamma\in [\gamma]$ is a cone geodesic representing the equivalence class.
\end{nota}

The following result is well known in Lorentzian geometry \cite{Low89} and works analogously for strongly convex cone structures.

\begin{prop}
Let $(M,C)$ be a strongly causal strongly convex cone structure.
Then $\mathcal{N}_C$ has a (possibly non-Hausdorff) smooth structure, turning the quotient map $\partial C\rightarrow \mathcal{N}_C$ into a submersion.
\end{prop}

\begin{remark}
In the Lorentzian setting, the question weather $\mathcal{N}_C$ is Hausdorff has been explored for example in \cite{Low89, Low90, Hedicke20}.
Examples include globally hyperbolic spacetimes and causally simple spacetimes with an open conformal embedding into a globally hyperbolic spacetime.
Using the analogous notions for geodesics in Finsler spacetimes, see e.g. \cite{Javaloyes15}, one can similarly show that these results also hold in the context of strongly convex cone structures, see \cite{Hedickediss}.
\end{remark}

In this paper we will focus on globally hyperbolic cone structures.

\subsection{Contact manifolds}\label{SecCont}

In the case $\mathcal{N}_C$ is a smooth manifold, it naturally carries a contact structure.
In this section we briefly recall the basic notions of contact geometry.
For a more detailed treatment, see \cite{Geiges}.

\begin{definition}
A hyperplane distribution $\xi\subset TM$ on a manifold $N^{2n+1}$ is called a \textbf{contact structure}, if there exists a $1$-form $\alpha$ such that $\xi=\mathrm{ker}\alpha$ and $\alpha\wedge(d\alpha)^n$ is a volume form on $N$ ($\xi$ is maximally non integrable).
\end{definition}

In this case a $1$-form $\beta$ with $\xi=\mathrm{ker}\beta$ is called a \textbf{contact form} for $\xi$ and induces a coorientation of $\xi$.
Given a contact structure $\xi=\mathrm{ker}\alpha$, a $1$-form $\beta$ is a contact form for $\xi$ if and only if $\beta=f\alpha$ for some non-vanishing function $f$.

\begin{exmp}\label{exmpcotangent}
Let $\Sigma$ be a smooth manifold.
The cotangent bundle of $\Sigma$ admits a canonical Liouville $1$-form $\lambda$, defined at $v\in T^{\ast}\Sigma$ by $\lambda_v(w):=v(d\pi(w))$, where $\pi \colon T^{\ast}\Sigma\rightarrow \Sigma$ denotes the canonical projection map.
The kernel of $\lambda$ is invariant under the $\R$-action on $T^{\ast}\Sigma$ given by positive scalar multiplication and therefore descends to a well-defined hyperplane distribution $\xi$ on the \textbf{spherical cotangent bundle} $$ST^{\ast}\Sigma:=T^{\ast}\Sigma\setminus\{0\}/\R.$$
Note that $\lambda$ does not descend to a well-defined $1$-form on $ST^{\ast}\Sigma$.

One way to obtain a contact form defining $\xi$ is to identify $ST^{\ast}\Sigma$ with the unit cotangent bundle of a fixed Riemannian (or more generally Finsler) metric $g$ on $\Sigma$.
The Legendre transform of the unit disc bundle is given by $K:=\{g(w,\cdot)|\Vert w\Vert\leq 1\}$ and defines a fibre-wise convex domain with smooth boundary in $T^{\ast}\Sigma$.
The boundary is a section for the $\R$-action, hence the restriction of $\lambda$ to $\partial K$ defines a contact form on $ST^{\ast}\Sigma$, see \cite{Geiges}.
More generally, any contact form $\alpha$ (inducing the standard coorientation on $\xi$) can be obtained by restricting $\lambda$ to the boundary of a fibre-wise bounded and star shaped domain with smooth boundary $K_{\alpha}$ in $T^{\ast}\Sigma$ centred around the zero section, i.e., a domain whose boundary is transverse to the flow of the Liouville vector field $Y$ inducing the $\R$-action.
\end{exmp}

\begin{definition}
A \textbf{contactomorphism} of $(M,\xi)$ is a diffeomorphism $\phi\colon M\rightarrow M$ with $\phi_{\ast}\xi=\xi$.
The set of all contactomorphisms forms a group denoted by $\mathrm{Cont}(M,\xi)$.
\end{definition}

An important class of submanifolds in a contact manifolds are called Legendrians.

\begin{definition}\label{deflegpositive}
A \textbf{Legendrian} $L^n\subset (M^{2n+1},\xi)$ is an embedded $n$-dimensional submanifold with $TL\subset\xi$.
We say that two Legendrians $L_0,L_1$ are \textbf{Legendrian isotopic} if there exists a smooth isotopy $\rho\colon [0,1]\times L\rightarrow M$ such that $\rho_0(L)=L_0, \rho_1(L)=L_1$ and $\rho_t(L)$ is Legendrian for every $t\in [0,1]$.
The equivalence class of all Legendrians isotopic to a given Legendrian $L$ is denoted by $\mathrm{Leg}(L)$.
\end{definition}

Due to the maximal non integrability, the maximal dimension a submanifold tangent to $\xi$ can have is $n$.

\begin{exmp}
Consider the spherical cotangent bundle $\pi\colon ST^{\ast}\Sigma\rightarrow \Sigma$ of a smooth manifold $\Sigma$, equipped with the contact structure defined above.
Then each fibre $F_p:=\pi^{-1}(p)$ is a Legendrian sphere in $ST^{\ast}\Sigma$.
The fibres over two points $p_0,p_1\in \Sigma$ are Legendrian isotopic (one can e.g. lift a curve between $p_0$ and $p_1$ in $\Sigma$ to a Legendrian isotopy in $ST^{\ast}\Sigma$).
\end{exmp}

An interesting feature of the group of contactomorphisms is that it carries a natural invariant cone structure (infinite dimensional) and even a Lorentz-Finsler metric, see \cite{Eliashberg00, Abbondandolo22}.
The tangent space at the identity of $\mathrm{Cont}(M,\xi)$ consists of all \textbf{contact vector fields}, i.e., vector fields whose flow preserves the contact structure $\xi$.
Then the cone at the identity is given by all contact vector fields $X$ that are non-negatively transverse to $\xi$, i.e., such that $\alpha(X)\geq 0$ for some contact form inducing the chosen coorientation.

Following \cite{Eliashberg00} we define:

\begin{definition}
Let $(M,\xi=\ker\alpha)$ be a cooriented contact manifold.
A smooth path of contactomorphisms $(\phi_t)_{t\in[0,1]}$ is called \textbf{positive (non-negative)} if 
$$\alpha\left(\frac{d}{dt}|_{t=s}\phi_t(p)\right)>0\; (\geq 0)$$
for all $p\in M$ and $s\in [0,1]$.
Similarly, call a path of Legendrians $(L_t)_{t\in[0,1]}$ positive (non-negative) if there exists a Legendrian isotopy $\rho\colon [0,1]\times L\rightarrow M$ with $\rho_t(L)=L_t$ such that
$$\alpha\left(\frac{d}{dt}|_{t=s}\rho_t(x)\right)>0\; (\geq 0)$$
for all $x\in L$ and $s\in [0,1]$.
\end{definition}

\begin{remark}
Note that this definition does not depend on the particular choice of the contact form $\alpha$ but only on the induced coorientation of $\xi$.
\end{remark}

\begin{exmp}\label{exmpReeb}
Let $\alpha$ be a contact form.
It follows from the maximal non-integrability that there exists a unique vector field $R$, called the \textbf{Reeb vector field} of $\alpha$ such that $\alpha(R)=1$ and $d\alpha(R,\cdot)=0$.
A simple computation shows that its flow $R_t$ preserves the contact form $\alpha$ and is therefore by definition a positive path of contactomorphisms.
Given any Legendrian $L$, the path $L_t:=R_t(L)$ is positive.

In the case of the spherical cotangent bundle it is a well-known fact \cite{Geiges} that the Reeb flow of the contact form obtained by restricting the Liouville form to the unit cotangent bundle of a (Riemann) Finsler metric (see Example \ref{exmpcotangent}) coincides with the cogeodesic flow, i.e., with the Legendre transform of the geodesic flow, and the flow lines project to geodesics on the base manifold.
\end{exmp}

\subsection{The contact structure on the space of cone geodesics}

It was already observed by Roger Penrose in his early works on twistor theory and worked out in detail by Low, see e.g. \cite{Penrose722, Low06}, that the space of null geodesics admits a natural contact structure.
In this section we will review the construction of this contact structure in the setting of strongly convex cone structures.

First note that any cone structure $C$ induces a dual cone structure $C^{\ast}\subset T^{\ast}M$ by defining
$$C^{\ast}:=\{v\in T^{\ast}M|v(w)\geq 0 \forall w\in C\}.$$
A Lorentz-Finsler metric $L\colon TM\rightarrow \R$ defining $C$ induces a Legendre transform, i.e., a bundle isomorphism $\Phi_L\colon TM\rightarrow T^{\ast}M, w\mapsto g_w(w,\cdot)$ that maps $C$ to $C^{\ast}$.
The Legendre transform can be used to consider the push-forward of the geodesic spray of $L$ to $T^{\ast}M$, which induces a foliation of $\partial C^{\ast}$ by curves projecting to cone geodesics.
This allows to consider $\mathcal{N}_C$ as a quotient of $\partial C^{\ast}$.
Note that similar to the dual construction on $TM$ this identification does not depend on the choice of $L$.

\begin{prop}\label{propcs}
Let $(M,C)$ be a strongly causal strongly convex cone structure such that $\mathcal{N}_C$ is a smooth (Hausdorff) manifold and the projection map $\pi\colon \partial C^{\ast}\rightarrow \mathcal{N}_C$ a submersion.
Then there exists a canonical contact structure $\xi_C$ on $\mathcal{N}_C$ defined by $\xi_C= d\pi(\ker\lambda\cap T\partial C^{\ast})$.
Here $\lambda$ denotes the canonical Liouville $1$-form on $T^{\ast}M$.
The contact structure $\xi_C$ has a coorientation induced by $\lambda$.
\end{prop}

\begin{proof}
We need to show the existence of a contact form whose kernel is given by $d\pi(\ker\lambda\cap T\partial C^{\ast})$.
Let $L$ be a Lorentz-Finsler metric defining $C$.

First note that the cogeodesic flow of $L$ on $\partial C^{\ast}$ preserves the canonical symplectic $2$-form $\omega=d\lambda$, i.e., $\omega$ induces a well-defined $2$-form $\Omega$ on the quotient $A:=\partial C^{\ast}/\sim$ of $\partial C^{\ast}$ by the cogeodesic flow.
By assumption $A$ is a smooth Hausdorff manifold that can be identified with the space of parametrized null geodesics of $L$.

The $2$-form $\Omega$ is symplectic:
By construction $\Omega$ is closed, i.e., it remains to show that $\Omega$ is non-degenerate.
Consider $[u]\in A$ and $[v]\in TA$ with $[v]\neq 0$.
The tangent vector $[v]$ can be identified with an equivalence class of tangent vectors in $T\partial C^{\ast}$ that are related by the cogeodesic flow.
To show that $\Omega$ is non-degenerate, we need to find a $[w]\in TA$ with $\Omega_{[u]}([v],[w])\neq 0$.
This is the case if and only if there exist representatives $v,w\in T_u\partial C^{\ast}$ with $\omega(v,w)\neq 0$.
In particular, it suffices to show that $\ker \omega|_{T_u\partial C^{\ast}}$ projects to $0$ in $T_{[u]}A$.
A standard fact from symplectic linear algebra implies that $\ker \omega|_{T_u\partial C^{\ast}}$ is one-dimensional ($T_u\partial C^{\ast}$ is a $2n-1$ dimensional subspace of $T_uT^{\ast}M^{n}$).
Let $X_L$ be the cogeodesic vector field.
Then for any $w\in T_u\partial C^{\ast}$ we have 
$$\omega_u(X_L(u),w)=-dH(w)=0,$$
where $H\colon T^{\ast}M\rightarrow\R$ with $H(g_w(w,\cdot))=L(w)$ denotes the Hamiltonian function induced by $L$ and we used that $\partial C^{\ast}$ is a connected component of $H^{-1}(0)$.
It follows that $\ker \omega|_{T_u\partial C^{\ast}}$ is spanned by $X_L(u)$ and projects to $0$ in $A$.

The canonical Liouville vector field $Y$, uniquely defined by $\omega(Y,\cdot)=\lambda$, induces the above $\R$-action on $T^{\ast}M$ and is tangent to $\partial C^{\ast}$.
An easy computation shows that the commutator of $Y$ and the cogeodesic vector field is co-linear to the geodesic vector field.
In particular, $Y$ induces a well-defined vector field $Z$ on $A$.

Let $\eta:=\Omega(Z,\cdot)$.
Similar to the above considerations one shows that $d\eta=\Omega$, i.e., that $\eta$ is a Liouville form for $\Omega$ and $Z$ a Liouville vector field.

Further note that $\mathcal{N}_C$ can be identified with the quotient of $A$ by the flow of $Z$.
As $Y$ and hence also $Z$ is a complete vector field, it induces the structure of a trivial principle $\R$-bundle $\pi\colon A\rightarrow \mathcal{N}_C$.
The manifold $\mathcal{N}_C$ is diffeomorphic to a section $S$ of that bundle.
Moreover, the vector field $Z$ is transverse to $S$.
By \cite[Lemma/Definition 1.4.5]{Geiges} the $1$-form $\alpha:=\eta|_{TS}$ is a contact form.
Hence it induces a contact form with the desired properties on $\mathcal{N}_C$.
\end{proof}

\begin{remark}
The contact form constructed above is not unique but depends on the choice of the section $S$.
In general there is no canonical contact form for the contact structure $\xi_C$ but only a canonical coorientation induced by the Liouville form on $T^{\ast}M$.
\end{remark}

A particularly important example is given by globally hyperbolic cone structures.

\begin{prop}\label{Propglobhypcont}
Let $(M,C)$ be a globally hyperbolic strongly convex cone structure with Cauchy hypersurface $\Sigma$.
Then $(\mathcal{N}_C,\xi_C)$ is a smooth contact manifold and there exists a canonical contactomorphism
$$\rho_{\Sigma}\colon \mathcal{N}_C\rightarrow ST^{\ast}\Sigma.$$
\end{prop}

\textbf{Warning:} The contactomorphism $\rho_{\Sigma}$ reverses the coorientation, i.e., the natural coorientation of $\xi_C$ is mapped to the coorientation of $\xi$ induced by $-\lambda$. 
The reason is our convention for the definition of $C^{\ast}$ that we chose in order to relate causality to positivity in the space of cone geodesics in Proposition \ref{proppositivecausal} below.

\begin{proof}
Let $L$ be a Lorentz-Finsler metric defining $C$.
By definition $\partial C^{\ast}|_{\Sigma}$ is a section of the cogeodesic flow of $L$ on $\partial C^{\ast}$.
In particular, $\partial C^{\ast}|_{\Sigma}$ is diffeomorphic to the space $A$ defined in the proof of Proposition \ref{propcs}.
Note that the canonical Liouville vector field $Y$ on $T^{\ast}M$ is tangent to  $\partial C^{\ast}|_{\Sigma}$ and induces an $\R$-action.
The space of cone geodesics $\mathcal{N}_C$ can be identified with the quotient $\partial C^{\ast}|_{\Sigma}/\R$ that induces the above topology.
The existence of the section $\partial C^{\ast}|_{\Sigma}$ implies that this topology is Hausdorff.
It is easy to check that the quotient map $\partial C^{\ast}\rightarrow \partial C^{\ast}|_{\Sigma}/\R$ is indeed a submersion and hence induces the canonical smooth structure on $\mathcal{N}_C$.

It remains to show that there exists a diffeomorphism from $\partial C^{\ast}|_{\Sigma}/\R$ to $ST^{\ast}\Sigma$ mapping the projection of the kernel of the canonical Liouville form $\lambda$ on $T^{\ast}M$ to the standard contact structure on $ST^{\ast}\Sigma$.

The diffeomorphism is given by the map

\begin{align*}
\rho\colon \partial C^{\ast}|_{\Sigma}/\R&\rightarrow ST^{\ast}\Sigma\cong T^{\ast}\Sigma\setminus\{0\}/\R \\
[v]&\mapsto \left[v|_{T\Sigma}\right].
\end{align*}

The map $\rho$ is well-defined since the $\R$-actions on $ C^{\ast}|_{\Sigma}$ and $ T^{\ast}\Sigma\setminus\{0\}$ are induced by scalar multiplication.
One easily checks that $\rho$ is a diffeomorphism.
Since the canonical Liouville form on $T^{\ast}\Sigma$ is the restriction of the Liouville form on $T^{\ast}M$, the contact structure $d\pi(\ker\lambda\cap T\partial C^{\ast}|_{\Sigma})$ is mapped to the standard contact structure on $ST^{\ast}\Sigma$.
\end{proof}

\begin{remark}
A more explicit contactomorphism can be constructed as follows.
Let $L$ be a Lorents-Finsler metric inducing $C$ and $N$ a unit normal vector field to the Cauchy surface $\Sigma$, i.e., $L(N)=1$ and $g_w(w,N)=0$ for all $w\in T\Sigma$.
Then the contactomorphism $\rho$ is defined by $\rho([\gamma])=g_{\gamma'(t_0)}(\gamma'(t_0),\cdot)|_{T\Sigma}$, where $\gamma$ is parametrized such that $\gamma(t_0)\in\Sigma$ and $g_N(N,\gamma'(t_0))=1$.
\end{remark}

There are strong connections between the causality of a cone structure $(M,C)$ and the contact geometry of $(\mathcal{N}_C,\xi_C)$.
Let $p\in M$.
The \textbf{sky} of $p$ is defined as
$$S(p):=\{[\gamma]\in\mathcal{N}_C|\gamma \cap \{p\}\neq\emptyset\}$$
which is the set of all cone geodesics through the point $p$.

\begin{prop}\label{proppositivecausal}
Let $(M,C)$ be a strongly causal strongly convex cone structure such that $(\mathcal{N}_C,\xi_C)$ is a smooth contact manifold.
Then for any $p\in M$ the sky $S(p)$ is a Legendrian sphere in $\mathcal{N}_C$.
All skies are in the same Legendrian isotopy class $\mathcal{S}_C$.
If moreover $\gamma\colon I\rightarrow M$ is a timelike/ causal curve.
Then $(S(\gamma(t))_{t\in I}$ is a positive/ non-negative path of Legendrians.
\end{prop}

\begin{proof}
Let $p\in M$ and $U$ be a causally convex globally hyperbolic neighbourhood of $p$ with Cauchy surface $\Sigma\cong\R^n$ and $p\in\Sigma$.
Then the space of cone geodesics of $U$ is an open neighbourhood of $S(p)$ in $\mathcal{N}_C$.
By Proposition \ref{Propglobhypcont} this neighbourhood is contactomorphic to $ST^{\ast}\Sigma$.
Note that this contactomorphism maps $S(p)$ to the fibre $F_p\subset ST^{\ast}\Sigma$ over $p$.
Since $F_p$ is a Legendrian sphere, the same holds for $S(p)$.

Given two points $p,q\in M$ and a path $\gamma\colon I\rightarrow M$ between them, $S(\gamma(t))$ defines a Legendrian isotopy between $S(p)$ and $S(q)$.
Suppose that $\gamma(t)$ is timelike.
Let $\iota_t\colon S^n\rightarrow S(\gamma(t))$ be a parametrization of the isotopy.
We need to show that for every $x\in S^n$ the curve $\iota_t(x)\colon I\rightarrow \mathcal{N}_C$ is positively transverse to $\xi_C$ with respect to the coorientation induced by the canonical Liouville form $\lambda$ on $T^{\ast}M$.
As before we identify $\mathcal{N}_C$ with $\partial C^{\ast}/\sim$.
Fix $t_0\in I$ and $x\in S^n$.
For arbitrary $t\in I$ we have $\iota_{t}(x)=[v_t]$ for some $v_t\in \partial C^{\ast}$ with $\pi(v_t)=\gamma(t)$, where $\pi\colon T^{\ast}M\rightarrow M$ is the canonical projection.
Note that by the choice of coorientation induced by $\lambda$, the isotopy is positively transverse to $\xi_C$ at $\iota_{t_0}(x)$ if and only if $\lambda_{v_{t_0}}\left(\frac{d}{dt}|_{t=t_0}v_t \right)>0$.
But 
$$\lambda_{v_{t_0}}\left(\frac{d}{dt}|_{t=t_0}v_t \right)=v_{t_0}\left(d\pi\left(\frac{d}{dt}|_{t=t_0}v_t\right)\right)=v_{t_0}(\gamma'(t_0))>0.$$
Here we used that $\gamma'(t_0)\in \mathrm{int}(C)$ and $v_{t_0}\in \partial C^{\ast}$.
Analogously one shows that for causal curves $S(\gamma(t))$ is non-negative.
\end{proof}

\section{Positive paths and cone structures}

The aim of this section is to establish a correspondence between certain globally hyperbolic cone structures equipped with a Cauchy time function, and positive paths of contactomorphisms.

In Section 3.1 we will show that any strongly convex globally hyperbolic cone structure gives rise to a positive path of contactomorphisms in the spherical cotangent bundle of its Cauchy hypersurfaces.
This proves Theorem \ref{thm1}

In Section 3.2 we will see that positive paths of contactomorphisms on spherical cotangent bundles can be used to define possibly non-strongly convex cone structures with the structure of a locally Lipschitz Lorentz-Finsler space.
In this case the path of contactomorphisms provides a natural analogue of null geodesic flows for such cone structures.
This proves Theorem \ref{thm2}.

\subsection{From cone structures to positive paths}

Let $(M,C)$ be a globally hyperbolic strongly convex cone structure with smooth Cauchy time function $\tau\colon M\rightarrow \R$.
Let $Y$ be a future pointing timelike vector field with $d\tau(Y)=1$.
Note that this condition implies that $Y$ is complete since $\tau$ is a Cauchy time function, i.e., $\tau$ is unbounded along any causal curve.
We get a diffeomorphism $\R\times \Sigma\rightarrow M$, $(t,p)\mapsto f_t^Y(p)$, where $f_t^Y$ denotes the flow of $Y$.
The diffeomorphism pulls back $\tau$ to the $\R$-coordinate $t$ and $Y$ to $\partial_t$.
For simplicity, we will without loss of generality from now on assume that $M=\R\times\Sigma$, $\tau=t$ and $Y=\partial_t$.

\begin{nota}
We will denote $\Sigma=\Sigma_0=\{0\}\times\Sigma$ and $\Sigma_t=\{t\}\times\Sigma$.
Further recall that $(ST^{\ast}\Sigma,\xi)$ denotes the spherical cotangent bundle with its canonical contact structure introduced in Example \ref{exmpcotangent} and $F_p$ the fibre over the point $p\in \Sigma$.
We consider the standard coorientation of $\xi$ induced by the canonical Liouville form $\lambda$.
As before we denote the sky of a point $(t,p)\in M=\R\times\Sigma$ by $S(t,p)$.
\end{nota}

Proposition \ref{Propglobhypcont} implies that for any $t\in\R$ there exists a contactomorphism
\begin{align*}
\abb{\rho_t}{\mathcal{N}_C}{ST^{\ast}\Sigma_t}
\end{align*}

\begin{theorem}\label{thmpospath}
Let $(M,C)=(\real{}\times \Sigma, C)$ be strongly convex and globally hyperbolic.
Then the family 
\begin{align*}
\abb{\phi_t}{ST^{\ast}\Sigma}{ST^{\ast}\Sigma}\:; \quad \phi_t:=\rho_t\circ\rho_0^{-1}
\end{align*}
defines a positive path of contactomorphisms.
Here we use the canonical identification of $ST^{\ast}\Sigma=ST^{\ast}\Sigma_0$ and $ST^{\ast}\Sigma_t$ induced by the splitting $M\cong \R\times \Sigma$.
Moreover, given $v\in ST^{\ast}\Sigma$ the curve
$$\gamma_v(t):=(t,\pi(\phi_t(v)))$$
is a cone geodesics of $C$.
Here $\pi\colon ST^{\ast}\Sigma\rightarrow\Sigma$ denotes the canonical projection.
\end{theorem}

\begin{proof}
By definition $\phi_t$ is a contactomorphism.
The path $(\phi_t)_{t\in\R}$ is positive iff the curve $\phi_t(v)$ is positively transverse to $\xi$ for every $v\in ST^{\ast}\Sigma$.
Further note that the path $(\phi_t)_{t\in\R}$ is positive iff the inverse path $(\phi^{-1}_t)_{t\in\R}$ is negatively transverse to $\xi$, see equation (\ref{eq1}) below.

Given $v$ there exists $p\in \Sigma$ with $v\in F_p$.
Then $\phi^{-1}_t(F_p)=\rho_0(S(t,p))$.
Proposition \ref{proppositivecausal} implies that the Legendrian isotopy $t\mapsto S(t,p)$ is positive.
Since $\rho_0$ reverses the chosen coorientations, $\phi^{-1}_t(F_p)$ is a negative Legendrian isotopy and in particular negatively transverse to $\xi$ at $v$.

Let $[\eta_v]= \rho_0^{-1}(v)$ be the class in $\mathcal{N}_C$ identified with $v$ and $\eta_v$ a cone geodesic representing this class.
Without loss of generality we can assume that $\eta_v$ is parametrized such that $\eta_v(t)\in\Sigma_t$.
By definition of the contactomorphism $\rho_t$, the intersection of $\eta_v$ with $\Sigma_t$ is given by $(t,\pi(\rho_t([\eta_v])))=\gamma_v(t)$.
Hence $\eta_v=\gamma_v$.
\end{proof}

Recall that a (Riemann) \textbf{Finsler metric} on a manifold $\Sigma$ is a $1$-homogeneous function $F\colon T\Sigma\rightarrow\R$ such that $F^2$ satisfies all conditions of Definition \ref{deffinsler}, except that the fundamental tensor of $F^2$ is required to be positive definite, see e.g. \cite{Bao12} for further details.
Given $M\cong\R\times\Sigma$ and $C$ as above one can show.

\begin{lemma}\label{LemtrivialFinsler}
There exists a unique family of Finsler metrics $F_t$ on $\Sigma$ such that the (continuous) Lorentz-Finsler metric $L:=dt^2-F^2_t$ defines $C$, i.e., such that $\partial C=L^{-1}(0)\cap \{dt>0\}$.
\end{lemma}

\begin{proof}
The Lemma can be proved similarly to \cite[Theorem 2.17]{Javaloyes20}.
As pointed out in this reference, the set $((\{dt=1\}\cap C)-\partial_t)\subset T\Sigma_t$ is a compact, strongly convex neighbourhood of the zero-section, and therefore the unit disc bundle of a unique Finsler metric $F_t$ on $\Sigma_t$.
Clearly $v\in\partial C$ iff $L(v)=0$ and $dt(v)>0$.
\end{proof}

\begin{remark}
In general the function $L$ is only $C^1$ on $TM$ and a smooth Lorentz-Finsler metric outside the span of $\partial_t$.
In fact $L$ is smooth everywhere if and only if $F_t$ is induced by a family of Riemannian metrics, see \cite[Proposition 4.1]{Warner65}.
As $L$ is always smooth around the boundary of the cones $C$, its null geodesics are well-defined and coincide up to parametrization with the cone geodesics of $C$.
\end{remark}

The positive path $(\phi_t)_{t\in\R}$ can be expressed in terms of the null geodesic flow of the metric $L$ as follows.
Note that any path of contactomorphisms $(f_t)_{t\in\R}$ uniquely defines a time-dependent contact vector field $X_t^{f}$ by the formula
$$X_t^f( f_t^{-1}(p)):= \frac{d}{dt}f_t(p).$$
Conversely, any path of contactomorphisms is uniquely determined by integrating a given time dependent contact vector field.
The vector field of a path can be expressed in terms of the vector field of the inverse path $(f_t^{-1})_{t\in\R}$
\begin{align}\label{eq1}
X_t^{f}=-df_t(X^{f^{-1}}_t\circ f_t^{-1})
\end{align}
by deriving the equation $f_t\circ f_t^{-1}=\mathrm{id}$.
It can be seen from this formula that if $(f_t)_{t\in\R}$ is a positive path, then $(f^{-1}_t)_{t\in\R}$ is negatively transverse to the contact structure.

Further recall from Example \ref{exmpReeb} that any finsler metric on $\Sigma$ uniquely determines a contact form on $ST^{\ast}\Sigma$ whose Reeb vector field is given by the cogeodesic vector field by restricting the canonical Liouville form to the unit cotangent bundle of the Finsler metric.

\begin{theorem}\label{ThmReebPath}
Let $M\cong\R\times \Sigma$, $C$ and $L$ be as above.
Then
$$X_t^{\phi}=R_{\alpha_t}.$$
Here $R_{\alpha_t}$ denotes the Reeb vector field of the contact form $\alpha_t$ induced by the Finsler metric $F_t$ on $ST^{\ast}\Sigma$.
\end{theorem}

\begin{proof}

First note that the fact that $X_t^{\phi}$ is a contact vector field positively transverse to the contact structure, implies that there exists a unique family of contact forms $\beta_t$ such that at fixed time $t$ the vector field $X_t^{\phi}$ is the Reeb vector field of $\beta_t$.
The contact form $\beta_t$ is uniquely determined by the property $\beta_t(X_t^{\phi})=1$.
The condition $d\beta_t(X_t^{\phi},\cdot)=0$ is automatically satisfied since $X_t^{\phi}$ is a contact vector field, see Lemma \ref{LemPathReeb} below for details.
It remains to show that $\beta_t=\alpha_t$ is the contact form induced by the Finsler metric $F_t$.

As explained in Example \ref{exmpcotangent} the contact form $\alpha_t$ is constructed in the following way.
We identify $ST^{\ast}\Sigma$ with the set $\{g^t_w(w,\cdot)|F_t(w)=1\}$, where $g^t$ denotes the fundamental tensor of $F_t$.
Then at $v=g^t_w(w,\cdot)$ the contact form is given by $(\alpha_t)_v(u)=g^t_w(w,d\pi(u))$.
It further follows from the definition of the contactomorphism $\rho_t$ that for any $v\in ST^{\ast}\Sigma$ we have 
$\phi_t(v)=g^t_{\tilde{w}_t}(\tilde{w}_t,\cdot)$.
Here $g^t$ denotes the fundamental tensor of $F_t$ and $\tilde{w}_t=\gamma_v'(t)-\partial_t$ is the part of the derivative of  $\gamma_v(t)=(t,\pi(\phi_t(v)))$ that is tangent to $\Sigma$.
One computes
$$d\pi\left(X_t^{\phi}\left(\phi_t(v)\right)\right)=\frac{d}{dt}\pi(\phi_t(v))=\tilde{w}_t.$$
Hence for every $v\in ST^{\ast}\Sigma$
$$(\alpha_t)_{\phi_t(v)}\left(X_t^{\phi}\left(\phi_t(v)\right)\right)=g^t_{\tilde{w}_t}(\tilde{w}_t,\tilde{w}_t)=F_t(\tilde{w}_t)=1.$$
Here we used that $L(\gamma_v'(t))=0$, i.e., $F_t(\tilde{w}_t)=1$.
\end{proof}

\subsection{The cone structure induced by a positive path}

Proposition \ref{proppositivecausal} suggests that there is a strong connection between the concepts of causality for cone structures and positivity in contact manifolds.
In the following we will construct a cone structure associated to a positive path of contactomorphisms in a spherical cotangent bundle.
The construction can be motivated as follows.

Let $C$ be a strongly convex globally hyperbolic cone structure on $M=\R\times\Sigma$ and $(\phi_t)_{t\in\R}$ the positive path constructed in the previous section.
Consider a curve $\gamma(s)=(\gamma_0(s),\tilde{\gamma}(s))$ in $\real{}\times \Sigma$.
Proposition \ref{proppositivecausal} implies that $\gamma$ is future pointing causal at $s_0$ if and only if $S(\gamma(s))$ is non-negative near $s_0$, see Definition \ref{deflegpositive}.
Since $\rho_0$ is coorientation reversing this implies that $\rho_0(S(\gamma(s))$ is a Legendrian isotopy in $ST^{\ast}\Sigma$ that is non-positively transverse to $\xi$.
The sky $S(\gamma(s))$ can be written in terms of the fibres of $ST^{\ast}\Sigma$ as $\rho_0(S(\gamma(s)))=\phi^{-1}_{\gamma_0(s)}(F_{\tilde{\gamma}(s)})$.
Choose a Legendrian isotopy $\abb{\iota_s}{S^{n-1}}{ST^{\ast}\Sigma}$ with $\iota_s(S^{n-1})=F_{\tilde{\gamma}(s)}$.
Pick a contact form $\alpha$ on $ST^{\ast}\Sigma$ inducing the standard coorientation.
Further pick $u\in S^{n-1}$ and set $w=\phi^{-1}_{\gamma_0(s_0)}(\iota_{s_0}(u))$.
Then the non-positivity condition of the Legendrian isotopy at $u$ can be written as
\begin{align*}
&\alpha_{w}\left(\frac{d}{ds}|_{s=s_0}\phi^{-1}_{\gamma_0(s)}(\iota_s(u))\right)\\
&=\alpha_{w}\left(d\phi^{-1}_{\gamma_0(s_0)}\left(\frac{d}{ds}|_{s=s_0}\iota_s(u)\right)+\gamma_0'(s_0) X^{\phi^{-1}}_{\gamma_0(s_0)}(w)\right)\\
&=\left(\left(\phi^{-1}_{\gamma_0(s_0)}\right)^{\ast}\alpha\right)_{\iota_{s_0}(u)}\left(\hat{\gamma}'(s_0)\right)+\gamma_0'(s_0)\alpha_w\left( X^{\phi^{-1}}_{\gamma_0(s_0)}(w)\right)\leq 0.
\end{align*}
Here $\hat{\gamma}'(s_0)$ denotes any vector in $T_{\iota(s_0(u))}ST^{\ast}\Sigma$ with $d\pi(\hat{\gamma}'(s_0))=\tilde{\gamma}'(s_0)$.
This  is non-positive for any $u\in S^{n-1}$ if and only if
$$\max\limits_{v\in F_{\tilde{\gamma}(s_0)}}\left( ((\phi^{-1}_{\gamma_0(s_0)})^{\ast}\alpha)_{v}(\hat{\gamma}'(s_0))+\gamma_0'(s_0)\alpha_{\phi^{-1}_{\gamma_0(s_0)}(v)}\left( X^{\phi^{-1}}_{\gamma_0(s_0)}\left(\phi^{-1}_{\gamma_0(s_0)}(v)\right)\right)\right)\leq 0.$$
For an arbitrary positive path of contactomorphisms this motivates the following definition.
Given a positive path of contactomorphisms $(f_t)_{t\in\R}$ inducing the time dependent vector field $X_t^f$, there exists a unique family of contact forms $(\alpha^f_t)_{t\in\R}$ with $\alpha^f_t\left(X_t^f\right)\equiv1$.
Any two contact forms inducing the same contact structure $\xi$ are related by a non-vanishing function.
Given a fixed contact form $\alpha$ the positivity of $(f_t)_{t\in\R}$ implies that the contact forms 
$$\alpha^f_t:=\frac{1}{\alpha\left(X_t^f\right)}\alpha$$
are well-defined. 

\begin{lemma}\label{LemPathReeb}
For fixed time $t$ the contact vector field $X^{f}_t$ is the Reeb vector field of $\alpha^f_t$.
Moreover, $\alpha_t^f=(f_t^{-1})^{\ast}\hat{\alpha}_t^f$, where $\hat{\alpha}_t^f$ is the unique family of contact forms satisfying $\hat{\alpha}_t^f\left(X_t^{f^{-1}}\right)=-1$.
\end{lemma}

\begin{proof}
By definition we have $\alpha^f_t\left(X_t^f\right)=1$.
It remains to show that $X_t^f$ lies in the kernel of $d\alpha^f_t$.
Note that $f_s^{\ast}\alpha^f_t=h_{t,s}\alpha^f_t$ for some family of positive functions $h_{t,s}$.
In particular $\frac{d}{ds}f_s^{\ast}\alpha^f_t=\left(\frac{d}{ds}h_{t,s}\right)\alpha^f_t$.
Using Cartan's formula for the Lie-derivative of time-dependent vector fields \cite[Lemma B.2]{Geiges} one computes
\begin{align*}
f_t^{\ast}d\alpha^f_t\left(X_t^f,\cdot\right)= \left(\frac{d}{ds}|_{s=t}h_{t,s}\right)\alpha^f_t.
\end{align*}
Since $TST^{\ast}\Sigma$ decomposes into $\xi=\ker \alpha^f_t$ and the span of $X_t^f$ it follows that $d\alpha^f_t\left(X_t^f,\cdot\right)=0$.

The second claim follows from equation (\ref{eq1}) and the definition of the Reeb vector field.
\end{proof}

The above observation for globally hyperbolic cone structures together with Lemma \ref{LemPathReeb} motivates the following definition.

\begin{definition}\label{defconepath}
Let $(f_t)_{t\in\real{}}$ be a positive path of contactomorphisms on $ST^{\ast}\Sigma$.
For $(t,p)\in\real{}\times \Sigma$ define 
\begin{align*}
C_{f}(t,p)&:=\\
&\left\lbrace(w_0,w)\in T_{(t,p)}(\real{}\times \Sigma)|\max\limits_{v\in F_{p}} (\alpha^f_t)_{v}(\hat{w})-w_0\leq 0\right\rbrace
\end{align*}
and
$$C_{f}:=\bigcup\limits_{(t,p)\in\real{}\times \Sigma}C_{f}(t,p).$$
As before $\hat{w}$ denotes any vector in $T_vST^{\ast}\Sigma$ with $d\pi(\hat{w})=w$.
\end{definition}

Recall from Example \ref{exmpcotangent} that given a contact form $\alpha$ on $ST^{\ast}\Sigma$ this contact form is induced by a fibre-wise star-shaped set $K_{\alpha}\subset T^{\ast}\Sigma$ with smooth boundary centred around the zero section.

\begin{nota}
Given a positive path $(f_t)_{t\in\R}$, we denote with $K_t^f:=K_{\alpha_t^f}\subset T^{\ast}\Sigma$ the fibre-wise star-shaped set inducing the contact forms $\alpha_t^f$ defined above.
\end{nota}

Given a subset $A\subset (\real{n})^{\ast}$ its \textbf{polar set} is defined as
$$A^{\circ}:=\{w\in \real{n}|v(w)\leq 1 \forall v\in A\}.$$

\begin{theorem}\label{lempolar}
Let $\abb{f_t}{ST^{\ast}\Sigma}{ST^{\ast}\Sigma}$ be a positive path of contactomorphisms.
Then $C_{f}$ is the cone defined by the fibre-wise polar of $K_t^f$, i.e.,
$$C_{f}=\bigcup\limits_{s>0}s\left(\partial_t+(K_t^f)^{\circ}\right).$$
In particular, $C_f$ is a proper closed cone structure.
\end{theorem}

\begin{proof}
By the definition of $K^f_t$ we have that 
$$\max\limits_{v\in F_{p}}\alpha_t^f(\hat{w})+w_0=\max\limits_{v\in K^f_t(p)}v(w)+w_0.$$
Thus $\partial_t+w\in C_{f}(t,p)$ if and only if 
\begin{align*}
 \max\limits_{v\in K_t^f(p)}v(w)\leq 1
\end{align*}

The polar of a set coincides with the polar of its convex hull and is in particular convex (see e.g. \cite{Schneider14}).
Thus $C_{f}(t,p)$ is a proper cone for every $(t,p)\in\real{}\times \Sigma$, i.e., $C_{f}$ is a proper closed cone structure.
\end{proof}

In general the cone structure $C_f$ is not strongly convex.
However, it is naturally induced by a function $G_f$ such that $C_f=G_f^{-1}([0,\infty))$.

\begin{definition}
Let $\abb{f_t}{ST^{\ast}\Sigma}{ST^{\ast}\Sigma}$ be a positive path of contactomorphisms.
Define
\begin{align*}
&\abb{G_{f}}{T(\R\times\Sigma)}{\real{}}\\
&(w_0,w)\mapsto w_0-\max\limits_{v\in F_{p}}\alpha_t^f(\hat{w})=w_0-\max\limits_{v\in K_t^f(p)}v(w).
\end{align*}
\end{definition}

In \cite[Section 2.13]{Minguzzi192} Minguzzi defined locally Lipschitz Lorentz-Finsler spaces, which provide a weaker notion of Lorentz-Finsler metrics suitable for non-strongly convex cone structures.
Recall from this reference that a cone structure $(M,C)$ is \textbf{locally Lipschitz} iff the set-valued map $p\mapsto C(p)$ is locally Lipschitz in the following way.
Let $p\in M$ and $U$ be an open neighbourhood diffeomorphic to $\R^{n+1}$.
The unit tangent bundle of $TU$ with respect to some auxilliary Riemannian metric admits a trivialization $STU\cong R^{n+1}\times S^n$.
Then $C|_U$ induces a map $q\mapsto C(q)\cap \{q\}\times S^n$ from $U$ to the space of compact subsets of $S^n$.
The cone structure $C$ is locally Lipschitz if and only if for every $p$ there exists a neighbourhood $U$ like above such that map from $U$ to the space of compact subsets of $S^n$ induced by $C$ is Lipschitz with respect to the Riemannian metric on $U$ and the Hausdorff-distance $d_H$ on the space of compact subsets of $S^n$.
Equivalently one can choose $U$ small enough such that the first coordinate $x_0$ in $\R^{n+1}$ satisfies $\partial_{x_0}\in \mathrm{int}(C_U)$ and define a map from $U$ to the space of bounded convex subsets of $\R^n$ by looking at $C(q)\cap \{dx_0=1\}$.

\begin{definition}
A locally Lipschitz \textbf{Lorentz-Finsler space} $(M,C,G)$ is a proper cone structure $(M,C)$ together with a positively $1$-homogeneous concave map $\abb{G}{C}{\real{}_{\geq 0}}$ such that the following holds.
The set
$$C^{\times}:=\{(w_0,w)\in T(\R\times M)|w\in C \text{ and } |w_0|\leq G(w)\}$$
is a locally Lipschitz proper closed cone structure on $\R\times M$.
\end{definition}

\begin{theorem}\label{thmfinslerspace}
Let $\abb{f_t}{ST^{\ast}\Sigma}{ST^{\ast}\Sigma}$ be a positive path of contactomorphisms.
Then $(\real{}\times \Sigma,C_f, G_{f})$ is a locally Lipschitz Lorentz-Finsler space.
\end{theorem}

\begin{proof}
Clearly $G_{f}$ is positive homogeneous and concave by the properties of the maximum.
It remains to show that $ (s,t,p)\mapsto C_{f}^{\times}(s,t,p)$ is locally Lipschitz as a set valued map from $\real{2}\times \Sigma$ to $T(\real{2}\times \Sigma)$.

By the construction of $C_f$, the coordinate $t$ of the second $\R$ factor satisfies that $\partial_t\in \mathrm{int}(C_{f}^{\times})$.
It follows from the definitions that 

\begin{align*}
C_{f}^{\times}(s,t,p)\cap \{dt=1\}&=\left\lbrace(w_0,w_1,w)\in T(\real{2}\times\Sigma)\left|1-\max\limits_{v\in K_t^f(p)}v(w)\geq |w_0|\right.\right\rbrace\\
&=\left\lbrace(w_0,w_1,w)\in T(\real{2}\times\Sigma)\left|\max\limits_{v\in K_t^f(p)}v(w)+|w_0|\leq 1\right.\right\rbrace\\
&=A^{\circ}(s,t,p),
\end{align*}
where $A^{\circ}(s,t,p)\subset T_{(s,t,p)}(\R^2\times\Sigma)$ is the polar of the set 
$$A(s,t,p):=\bigcup_{\lambda\in[-1,1]}\left(K_t^f(p)+\lambda ds\right).$$

Since $\alpha^f_t$ is a smooth family of contact forms, the map 
$$(s,t,p)\mapsto \bigcup\limits_{\lambda\in[-1,1]}( K_t^f(p)+\lambda ds) $$
is locally Lipschitz with respect to the Hausdorff distance.
The same holds for the convex hull operator \cite[Theorem 3.2.10]{Moszynska}.
Since the polar of a set coincides with the polar of its convex hull and the polar operator $K\mapsto K^{\circ}$ is locally Lipschitz on the space of convex bodies containing the origin \cite[Theorem 13.3.4]{Moszynska}, the result follows.
\end{proof}

A natural  question arising from Theorem \ref{thmfinslerspace} is how these results relate to the constructions of the previous section.
In fact one can show the following correspondence.

\begin{theorem}\label{Thmcorrespondence}
Let $(\R\times\Sigma, C)$ be a strongly convex globally hyperbolic cone structure inducing the positive path $\phi_t$ as in Theorem \ref{thmpospath} and $F_t$ the family of Finsler metrics constructed in Lemma \ref{LemtrivialFinsler}.
Then $C_{\phi}=C$ and $G_{\phi}=dt-F_t$.
\end{theorem}

\begin{proof}
We only need to show that $G_{\phi}=dt-F_t$, since 
$$C=\{dt^2-F_t^2\geq 0 \}\cap \{dt>0\}=\{dt-F_t\geq 0 \}.$$
For $(w_0,w)\in T^{\ast}_{(t,p)}(\R\times\Sigma)$ we have that $G_{\phi}(w_0,w)= w_0-\max\limits_{v\in F_{p}} (\alpha^{\phi}_t)_{v}(\hat{w})$, where $\alpha^{\phi}_t$ is the unique family of contact forms with $\alpha_t^{\phi}\left(X_t^{\phi}\right)=1$.
It follows from Theorem \ref{ThmReebPath} that $\alpha_t^{\phi}$ is induced by $F_t$.
Thus
$$G_{\phi}(w_0,w)=w_0-\max\limits_{v\in F_{p}} (\alpha^{\phi}_t)_{v}(\hat{w})=w_0-\max\limits_{v\in K^{\phi}_{t}} v(w)=w_0-F_t(w).$$
\end{proof}

A natural question is to ask about the causality of the cone structure $C_f$ for an arbitrary positive path of contactomorphisms.
In general this question seems quite hard to answer since the boundary of $C_f$ need not to be smooth and can, depending on the fibre-wise star-shaped set $K_t^f$ look quite complicated.
The above considerations imply global hyperbolicity at least in the following situation.

\begin{definition}
We call a positive path of contactomorphisms $f_t$ strongly convex if for each $t\in\R$ the contact form $\alpha^f_t$, where $\alpha^f_t\left(X_t^{f}\right)=1$, is induced by a Finsler metric.
\end{definition}

\begin{theorem}
Let $(f_t)_{t\in\R}$ be a strongly convex positive path of contactomorphisms.
Then $C_f$ is globally hyperbolic and strongly convex and the positive path $(\phi_t)_{t\in\R}$ induced by $C_f$ coincides with $(f_t)_{t\in\R}$.
\end{theorem}

\begin{proof}
Let $F_t$ be the family of Finsler metrics inducing the contact forms $\alpha^f_t$.
Then as in the proof of Theorem \ref{Thmcorrespondence} we have that
$$G_f(w_0,w)= w_0-\max\limits_{v\in F_{p}} (\alpha^f_t)_{v}(\hat{w})=w_0-F_t(w).$$
Hence $G_f=dt-F_t$ and $C_f=\{dt^2-F_t^2\geq 0 \}\cap \{dt>0\}$ is strongly convex.
As proved in Lemma \ref{LemPathReeb} the vector field $X_t^f$ coincides at fixed time $t$ with the Reeb vector field of $\alpha^f_t$ and hence with the cogeodesic flow of $F_t$.
The same holds for the path $(\phi_t)_{t\in\R}$ induced by $C_f$ via the space of cone geodesics by Theorem \ref{ThmReebPath}.
Note that by the strong convexity every cone geodesic is a re-parametrization of a curve of the form $\gamma_v(t)=(t,\pi(f_t(v)))$.
In particular all cone geodesics intersect the acausal hypersurface $\{0\}\times\Sigma$.
Analogously to \cite[Proposition 5.14]{Penrose72} one can show that $\{0\}\times\Sigma$ is a Cauchy hypersurface, i.e., that $C_f$ is globally hyperbolic.
\end{proof}

In view of this theorem it seems plausible that the cone geodesics of $C_f$ for more general paths of contactomorphisms all intersect $\{0\}\times\Sigma$.
Hence we conjecture the following.

\begin{conjecture}
Let $(f_t)_{t\in\R}$ be a positive path of contactomorphisms of $ST^{\ast}\Sigma$.
Then $(\R\times\Sigma,C_f)$ is a globally hyperbolic cone structure.
\end{conjecture}

\begin{remark}
The global hyperbolicity seems to result from the fact that $(f_t)_{t\in\R}$ is a well-defined path of contactomorphisms, meaning that the flow of the time dependent vector field $X_t^f$ is complete.
Most of the above constructions such as the cone structure $C_f$, which only depends on $X_t^f$, can also be considered for positive time-dependent contact vector fields with non-complete flows.
In these cases one can not expect to obtain a globally hyperbolic cone structure.
If for example $X_t=X$ is the Reeb-vector field induced by the geodesic spray of a Riemannian metric $g$, the cone structure obtained from $X$ is the one of the Lorentzian metric $dt^2-g$ which is globally hyperbolic if and only if $g$ is complete.
\end{remark}

\end{document}